\documentclass[12pt]{amsart}
\usepackage{a4}
\usepackage[T1]{fontenc}
\usepackage[utf8]{inputenc}
\usepackage[english]{babel}
\usepackage{amsmath}
\usepackage{amsthm}
\usepackage{amssymb}

\newtheorem{thm}{Theorem}%

\newtheorem{lem}[thm]{Lemma}

\newtheorem{cor}[thm]{Corollary}

\theoremstyle{definition}
\newtheorem{de}[thm]{Definition}

\numberwithin{equation}{section}

\newcommand{\N}{\mathbb{N}}
\newcommand{\Z}{\mathbb{Z}}
\newcommand{\C}{\mathcal{C}}

\newcommand{\Clo}{\operatorname{Clo}}
\newcommand{\Con}{\operatorname{Con}}
\newcommand{\Pol}{\operatorname{Pol}}
\newcommand{\Comp}{\operatorname{Comp}}
\usepackage{bm}
\newcommand{\ab}[1]{{\mathbf{#1}}} 
\newcommand{\ob}[1]{{\mathbb{#1}}} 
\newcommand{\vb}[1]{\bm{#1}} 
\newcommand{\algop}[2]{( {#1}, {#2} )}
\newcommand{\VecTwo}[2]{
   \big(
   \begin{smallmatrix}
      #1 \\ #2
   \end{smallmatrix}
   \big)
   }
\newcommand{\id}{\mathrm{id}}
\newcommand{\join}{\vee}
\newcommand{\meet}{\wedge}
\newcommand{\Int}[2]{{\ob{I}\, [{#1},{#2}]}}
\date{\today}

\title{Congruence preserving expansions of nilpotent algebras}

\author{Erhard Aichinger}
\address{
Institut für Algebra,
Johannes Kepler Universit\"at Linz, Altenberger Strasse 69, 4040 Linz,
Austria}
\email{erhard@algebra.uni-linz.ac.at}

\author{Gábor Horváth}
\address{Institute of Mathematics,
University of Debrecen,
Pf.~400, Debrecen, 4002,
Hungary}
\email{ghorvath@science.unideb.hu}

\subjclass[2010]{}
\thanks{
  The first listed author was supported by the Austrian Science Fund (FWF):~P24077 and~P29931.
 The second listed author was supported by the J\'{a}nos Bolyai Research Scholarship of the Hungarian Academy of Sciences,
by the European Union's Seventh Framework Programme (FP7/2007-2013) under grant agreement no.~318202, by the Hungarian Scientific Research Fund (OTKA) grant no.~K109185, and by the National Research, Development and Innovation Fund of Hungary, financed under the FK~124814 funding scheme. 
}

\begin{document}

\begin{abstract}
  We characterize those nilpotent algebras of prime power order
  and finite type
  in congruence modular varieties
  that have infinitely many polynomially inequivalent
  congruence preserving expansions. 
\end{abstract}

\maketitle

\section{The result}

Associated with every algebraic structure $\ab{A}$, there
are two clones which we will study in the present note:
the clone of polynomial functions $\Pol (\ab{A})$, and the
clone of congruence preserving functions $\Comp (\ab{A})$.
We say that an algebra $\ab{A'}$, defined on the same universe
as $\ab{A}$, is a \emph{congruence preserving expansion} of
$\ab{A}$ if $\Pol (\ab{A}) \subseteq \Pol(\ab{A'}) \subseteq \Comp (\ab{A})$.
For expanded groups, such expansions with \emph{unary} operations
have been studied in \cite{Pe:CEON}.
Considering algebras with the same clone of polynomial functions
as equivalent, we say that \emph{$\ab{A}$ has finitely many
  polynomially inequivalent
  congruence preserving expansions} if the set
$\{ C \mid C \text{ is a clone with } \Pol(\ab{A}) \subseteq
C \subseteq \Comp (\ab{A})\}$ is finite.
One extreme case
is $\Pol(\ab{A}) = \Comp (\ab{A})$: then $\ab{A}$ is called
\emph{affine complete} \cite{KP:PCIA}, and clearly $\ab{A}$ then has
only one congruence preserving expansion.
On the other side, if $\ab{A}$ has only finitely many fundamental operations
(i.~e., it is of \emph{finite type})  and $\Comp (\ab{A})$
is not finitely generated, then $\ab{A}$ has infinitely many inequivalent
congruence preserving expansions.
For finite $p$-groups $\ab{G}$, \cite{ALM:FGOC} provides a  complete
characterization 
when $\Comp (\ab{G})$ is finitely generated. However, there are
algebras for which $\Comp (\ab{A})$ is finitely generated, but $\ab{A}$
still has infinitely many inequivalent
congruence preserving expansions: the cyclic
group with $4$ elements \cite{Bu:PCCT} and the quaternion group with $8$ elements are
examples of such a behaviour. Our characterization
uses a condition on the congruence lattice that has been used in
\cite{AM:SOCO}: We say that a bounded lattice $\ob{L}$ \emph{splits strongly}
if it is the union of two proper subintervals $\ob{I}[0, \delta] \cup
 \ob{I}[\varepsilon, 1]$ with nonempty
intersection, which can be expressed by
\begin{equation} \label{eq:splitsstrongly}
\ob{L} \models \exists \, \delta, \varepsilon \in \ob{L}  \colon 
\big(
0 < \varepsilon \le \delta < 1 \text{ and }
(\forall \alpha \in \ob{L}  \colon  \alpha \le \delta \text{ or }
\alpha \ge \varepsilon) \big).
\end{equation}
Note that it is claimed that $\varepsilon \neq 0$ and $\delta \neq 1$. 
We say that a finite algebra $\ab{A}$ has \emph{few subpowers} if
there is a polynomial $p \in \mathbb{R}[x]$ such that for each $n \in \N$,
the algebra $\ab{A}^n$ has at most $2^{p(n)}$ subalgebras.
In \cite{BI:VWFS}, it is proved that such algebras
are characterized by having an \emph{edge term} and that
they generate congruence modular varieties.

The following theorem is the main result of the present note.
\begin{thm} \label{thm:1}
  Let $\ab{A}$ be a finite algebra of finite type with few subpowers.
  Then the following are equivalent:
  \begin{enumerate}
    \item \label{it:t0}  $\ab{A}$ has infinitely many polynomially inequivalent congruence preserving expansions.
    \item \label{it:t1} The interval $\mathcal{C} := \{ C  \mid C \text{ is a clone with } \Pol(\ab{A}) \subseteq
      C \subseteq \Comp (\ab{A}) \}$ in the lattice of all clones on $A$
      is infinite.
    \item \label{it:t2} There exists a clone $C$ with $\Pol (\ab{A}) \subseteq C \subseteq
      \Comp (\ab{A})$ that is not finitely generated.
  \end{enumerate} 
  If $\ab{A}$ is furthermore isomorphic to a direct product
  $\ab{A}_1 \times \cdots \times \ab{A}_n$ of nilpotent algebras of prime power order,
  and if for all $i,j \in \{1, \ldots, n \}$ with $i \neq j$, we have
  $\gcd (|\ab{A}_i|, |\ab{A}_j|) = 1$, 
  then the three conditions \eqref{it:t0}--\,\eqref{it:t2} are equivalent to  
  \begin{enumerate}
    \setcounter{enumi}{3}
  \item \label{it:t3} There is $i \in \{1, \ldots, n \}$ such that the congruence lattice of $\ab{A}_i$ splits strongly.
  \item \label{it:t4} The congruence lattice of $\ab{A}$ splits strongly.  
  \end{enumerate}
\end{thm}
The condition on the congruence lattice in items~\eqref{it:t3} and~\eqref{it:t4} has also appeared
in Theorem~1.1 of  \cite{AM:SOCO}, which states that a finite modular lattice
that splits strongly allows infinitely many different sequences satisfying
the properties of higher commutator operations.
Theorem~\ref{thm:1}  provides a description of those nilpotent groups that
have infinitely many polynomially inequivalent expansions.
\begin{cor} \label{cor:pgroup}
  Let $G$ be a finite nilpotent group.
  Then $G$ has infinitely many polynomially inequivalent
  congruence preserving expansions if and only if the lattice
  of normal subgroups of $G$ splits strongly, i.e., the normal
  subgroup lattice of $G$
  is the union of two proper subintervals that have
  at least one normal subgroup in common.
\end{cor}  

One can view the finiteness of the interval $\Int{\Pol (\ab{A})}{\Comp (\ab{A})}$
as a \emph{polynomial completeness property} \cite{KP:PCIA}. For finite abelian $p$-groups,
this property is described in the following corollary.
\begin{cor} \label{cor:abgroup}
  Let $p$ be a prime, let  $r \in \N$ and let $m_1, \ldots, m_r \in \N$ with
  $m_1 \ge \cdots \ge m_r$. Then the abelian group
  $G := \prod_{i=1}^r \Z_{p^{m_i}}$ has finitely many polynomially inequivalent
  congruence preserving expansions if and only if $(r \ge 2$ and $m_1 = m_2)$ or $(r=1$ and $m_1 = 1)$.
\end{cor}

We compare this to known completeness properties: $G = \prod_{i=1}^{r} \Z_{p^{m_i}}$
with $r \in \N$ and  $m_1 \ge \cdots \ge m_r \ge 1$ is \emph{affine complete}
if and only if  $r \ge 2$ and $m_1 = m_2$ \cite{No:UDAV}. By Corollary~\ref{cor:abgroup}, $G$ has
only \emph{finitely many polynomially inequivalent congruence preserving expansions}
if and only if it is affine complete or simple. Finally by \cite[Theorem~1.2]{ALM:FGOC},
the \emph{clone of congruence preserving functions of $G$ is finitely generated}
if and only if $G$ is affine complete or cyclic.

The proofs of the results stated in this introduction will be given
in Section~\ref{sec:proofs}. Before, Section~\ref{sec:dp} provides
some auxiliary results about clones
and direct products. 

\section{Direct products} \label{sec:dp}
We need some information on clones acting on direct products. The results
contained in this section are for the most part known, or follow quite
immediately from existing theory. In the sequel,
a vector $(a_1,\ldots, a_n)$ will sometimes by written as $\vb{a}$. 
\begin{de}
  Let $A, B$ be sets, let $n \in \N$, and let
  $c  \colon  A^n \to A$ and $d \colon  B^n \to B$. Then we define
   the mapping $c \otimes d  \colon  (A \times B)^n \to A \times B$
   by $$c \otimes d \,\,\, \left( \VecTwo{\vb{a}}{\vb{b}} \right) := \VecTwo{c (\vb{a})}{d(\vb{b})}$$
   for $\vb{a} \in A^n$, $\vb{b} \in B^n$.
\end{de}
For a clone $C$ on the set $A$ , we let $C^{[n]}$ be its $n$-ary part
$C \cap A^{A^n}$.
\begin{de}
   Let $A, B$ be sets, let $C$ be a clone on $A$, and let $D$ be
  a clone $B$. We define the set $C \otimes D$, which consists of finitary functions on
  $A \times B$,  by
  \[
  C \otimes D := \{ c \otimes d \mid 
                     n \in \N,  c \in C^{[n]},
  d \in D^{[n]} \}.
  \]
\end{de}

\begin{lem} \label{lem:cod}
  Let $A, B$ be sets, let $C$ be a clone on $A$, and let $D$ be
  a clone $B$.Then the set $C \otimes D$ is a clone on $A \times B$.
\end{lem}
\begin{proof}
$C \otimes D$ contains all projections. For
$f \in (C \otimes D)^{[n]}$ and $g_1, \ldots, g_n \in (C \otimes D)^{[m]}$,
straightforward calculations show  $f(g_1, \ldots, g_n) \in (C \otimes D)^{[m]}$.
\end{proof}

For a set $X$ of finitary functions on $A$, the \emph{clone generated by $X$}
is denoted by
$\Clo_A (X)$.
\begin{lem} \label{lem:X}
  Let $A, B$ be sets, let $C$ be a clone on $A$ that is generated
  by $X \subseteq C$, and let $D$ be a clone on $B$
  that is generated by $Y \subseteq D$.
  Let
  \begin{multline*}
  Z := \{ \VecTwo{\vb{a}}{\vb{b}} \mapsto \VecTwo{f (\vb{a})}{b_1} \mid
  f \in X \} \cup 
       \{ \VecTwo{\vb{a}}{\vb{b}} \mapsto \VecTwo{a_1}{g (\vb{b})} \mid
       g \in Y \} \,\, \cup \\
       \{ \left(\VecTwo{a_1}{b_1}, \VecTwo{a_2}{b_2}\right) \mapsto \VecTwo{a_1}{b_2} \}.
  \end{multline*}
  Then the clone on $A \times B$ that is generated by $Z$ is equal to
  $C \otimes D$.
\end{lem}
\begin{proof}
   We proceed as in the proof of Proposition~4.1 of \cite{ALM:FGOC}.
  We define  $\psi_C  \colon  C \otimes D \to C$ by $\psi_C (c \otimes d) = c$
  and
  $\psi_D  \colon   C \otimes D \to D$ by $\psi_D (c \otimes d) = d$.
  Adopting the viewpoint of \cite{Ma:IAAP},
  we consider clones as function algebras:
  the idea of this approach is that a clone on $A$ is a subalgebra of
  $\bigcup_{n \in \N} A^{A^n}$ equipped with the unary operations $\zeta$ (rotation of the arguments),
  $\tau$ (swapping the first two arguments), $\Delta$ (taking a minor), $\nabla$ (adding
  an inessential argument), and one binary operation $\circ$ that composes two functions in
  a certain way. A detailed account of this point of view is given in \cite[p.\ 38]{PK:FUR}.
  Using this approach, we observe that
  the mapping $\psi_C$ is an epimorphism
  from the algebra  $\algop{C \otimes D}{\id_{A \times B}, \zeta, \tau, \Delta, \nabla, \circ}$
  to the algebra  $\algop{C}{\id_{A}, \zeta, \tau, \Delta, \nabla, \circ}$.
  Since $X \subseteq \psi_C (Z)$,
  $\Clo_A (\psi_C (Z)) = C$.
   Now a basic property on the interaction of homomorphisms and subalgebra generation
  \cite[Theorem~II.6.6]{BS:ACIU} yields
   $\psi_C (\Clo_{A \times B} (Z)) = C$. Similarly, $\psi_D (\Clo_{A \times B} (Z)) = D$.
    We are now ready to show
    \begin{equation} \label{eq:clo}
      \Clo_{A \times B} (Z) = C \otimes D.
    \end{equation}
    The ``$\subseteq$''-inclusion follows from $Z \subseteq C \otimes D$.
    For ``$\supseteq$'', we choose $c \otimes d \in C \otimes D$.
    Since $\psi_C (\Clo_{A \times B} (Z)) = C$, we find $d' \in D$ such that
    $c \otimes d' \in \Clo_{A \times B} (Z)$. Similarly, we find $c' \in D$
    with $c' \otimes d \in \Clo_{A \times B} (Z)$. If we denote the binary projections
    by $\pi_1$ and $\pi_2$, we see that the last element listed in the definition
    of $Z$ is $\pi_1 \otimes \pi_2$. Hence the composition
    $\pi_1 \otimes \pi_2 \, (c \otimes d', c' \otimes d)$ lies in $\Clo_{A \times B} (Z)$,
    which implies $(c \otimes d) \in \Clo_{A \times B} (Z)$.
\end{proof}

\begin{cor} \label{cor:fg}
  Let $A, B$ be sets, let $C$ be a clone on $A$, and let $D$ be a clone
  on $B$. Then $C \otimes D$ is finitely generated if and only if both
  $C$ and $D$ are finitely generated.
\end{cor}
\begin{proof} The ``if''-direction follows from Lemma~\ref{lem:X}.
  For the ``only if''-direction, we observe that
  both function algebras 
  $\algop{C}{\id_{A}, \zeta, \tau, \Delta, \nabla, \circ}$
  and
  $\algop{D}{\id_{B}, \zeta, \tau, \Delta, \nabla, \circ}$
  are homomorphic images of the algebra
  $\algop{C \otimes D}{\id_{A \times B}, \zeta, \tau, \Delta, \nabla, \circ}$.
\end{proof}  

The polynomial functions on the direct product of two algebras can in general
not be determined directly from the polynomial functions on the factors (for finite  groups, this phenomenon has been studied in \cite{Sc:TAOP}).
However,
under the additional assumption that the algebras lie in a congruence permutable
variety and that all congruences in the direct product are
\emph{product congruences}, a decomposition into the direct factors
is possible.
Let $\ab{A} := \ab{B} \times \ab{C}$. A congruence $\alpha$ of
$\ab{A}$ is a \emph{product congruence} if there exist
$\beta \in \Con (\ab{B})$
  and $\gamma \in \Con (\ab{C})$ such that
  \[
  \alpha = \{ \left(\VecTwo{b_1}{c_1}, \VecTwo{b_2}{c_2}\right) \mid (b_1, b_2) \in \beta \text{ and }
  (c_1, c_2) \in \gamma \}.
  \]
  A congruence of $\ab{A}$ that is not a product congruence is a \emph{skew congruence}.
  We say that $\ab{A} = \ab{B} \times \ab{C}$
  is a \emph{skew-free direct product of $\ab{B}$ and $\ab{C}$}
  if $\ab{A}$ has no skew congruences.

\begin{lem}[{\cite{KM:PFOS}}] \label{lem:abc}
  Let $\ab{A}$ be an algebra with a Mal'cev term. Suppose that
  $\ab{A} = \ab{B} \times \ab{C}$  is a skew-free direct product of $\ab{B}$ and
  $\ab{C}$.
  Then $\Pol (\ab{A}) = \Pol (\ab{B}) \otimes \Pol (\ab{C})$.
 \end{lem}
\begin{proof}
  This is essentially Corollary~2 from \cite{KM:PFOS}; the claim
  can also be derived directly from Corollary~6.4 of \cite{AM:IOAW}.
\end{proof}  

\begin{cor}
    Let $\ab{A}$ be an algebra with a Mal'cev term. Suppose that
  $\ab{A} = \ab{B} \times \ab{C}$  is a skew-free direct product of $\ab{B}$ and
    $\ab{C}$. Then the interval between $\Pol(\ab{A})$ and $\Comp (\ab{A})$
    in the lattice of clones on $A$ is given by
    \begin{multline} \label{eq:CDP}
         \mathbb{I} [\Pol(\ab{A}), \, \Comp (\ab{A})]  \\ = 
    \{ E \otimes F \mid E \in \mathbb{I} [\Pol (\ab{B}), \, \Comp (\ab{B})], \,\,
    F \in \mathbb{I} [\Pol (\ab{C}),\,  \Comp (\ab{C})] \}.
    \end{multline}  
\end{cor}
\begin{proof}
  For $\subseteq$, 
  let $G$ be a clone with $\Pol (\ab{A}) \subseteq G \subseteq \Comp (\ab{A})$.
  Then $\ab{A'} := \algop{A}{G}$ has the same congruence lattice as $\ab{A}$
  and is therefore a skew-free direct product of two algebras $\ab{B'}$ and
  $\ab{C'}$. Now we let $E := \Pol (\ab{B'})$ and $F := \Pol (\ab{C'})$
  and use Lemma~\ref{lem:abc} to obtain $G = E \otimes F$.
  
  For $\supseteq$, we first observe that $\ab{A}$ is a skew-free
  direct product. This implies that every function in
  $\Comp (\ab{B}) \otimes \Comp (\ab{C})$ is a congruence
  preserving function on $\ab{A}$.
  Now we choose $E \otimes F$ from the right hand side of~\eqref{eq:CDP}.
  Then clearly $\Pol (\ab{A}) = \Pol (\ab{B}) \otimes \Pol (\ab{C})
  \subseteq E \otimes F \subseteq \Comp (\ab{B}) \otimes \Comp (\ab{C})
  \subseteq \Comp (\ab{A})$, and therefore $E \otimes F$ lies in the
  left hand side of~\eqref{eq:CDP}.
  \end{proof}

We will now investigate the splitting property of lattices that
appears in items~\eqref{it:t3} and~\eqref{it:t4} of Theorem~\ref{thm:1}.
\begin{lem} \label{lem:kl}
  Let $n \in \N$, and let $\ob{L}_1, \ldots, \ob{L}_n$ be bounded lattices,
  and let $\ob{K} := \ob{L}_1 \times \cdots \times \ob{L}_n$.
  Then $\ob{K}$ splits strongly if and only if at least one of the lattices $\ob{L}_i$
  splits strongly.
\end{lem}
\begin{proof}
For the ``if''-direction, assume that $\ob{L}_i$ splits strongly with
witnesses $\delta_i, \varepsilon_i$. Then $\ob{K}$ splits strongly
with $(1,\ldots, 1, \delta_i, 1, \ldots, 1)$  and
$(0,\ldots, 0, \varepsilon_i, 0, \ldots, 0)$ as witnesses ($\delta_i$ and
$\varepsilon_i$ at position $i$).

For the ``only if''-direction, we assume that
$\ob{K}$ splits strongly with witnesses $\vb{\delta} =
(\delta_1, \ldots, \delta_n)$ and
$\vb{\varepsilon} = (\varepsilon_1, \ldots, \varepsilon_n)$.
Since $\vb{\delta} \neq 1$, there is $i$ such that
$\delta_i < 1$. Hence $(0,\ldots, 0, 1, 0, \ldots, 0) \not\le \vb{\delta}$ (with $1$ at place $i$).
By the splitting property, we have $(0,\ldots, 0, 1, 0, \ldots 0) \ge \vb{\varepsilon}$. Thus
for $j \neq i$, we have $\varepsilon_j = 0$, and therefore, since
$\vb{\varepsilon} \neq 0$, we have $\varepsilon_i > 0$.
Now we show that $\ob{L}_i$ splits strongly with witnesses $\delta_i$ and $\varepsilon_i$:
since $\vb{\delta} \ge \vb{\varepsilon}$, we have
$\delta_i \ge \varepsilon_i$.
Now take any $\alpha \in \ob{L}_i$ with $\alpha \not\le \delta_i$.
Then $(0,\ldots, 0, \alpha, 0,\ldots, 0) \not\le \vb{\delta}$,
hence $(0,\ldots, 0, \alpha, 0,\ldots, 0) \ge  \vb{\varepsilon}$, and thus
$\alpha \ge \varepsilon_i$.
\end{proof}

A finite algebra is \emph{congruence uniform} if for every congruence
of $\ab{A}$, all its congruence classes have the same cardinality.
For $\alpha, \beta, \gamma, \delta \in \Con (\ab{A})$, we define
 \begin{equation} \label{eq:dequot}
     \#(\beta:\alpha) = |\ab{A}/\alpha| \, / \, |\ab{A}/\beta|
 \end{equation}
 and write $\Int{\alpha}{\beta} \nearrow \Int{\gamma}{\delta}$ (and also
  $\Int{\gamma}{\delta} \searrow \Int{\alpha}{\beta}$) if
 $\alpha = \beta \meet \gamma$ and $\beta \join \gamma = \delta$.
\begin{lem} \label{lem:uniform1}
  Let $\ab{A}$ be a finite congruence uniform algebra in a congruence
  permutable variety, and let $\alpha, \beta, \gamma, \delta \in \Con (\ab{A})$.
  Then we have:
  \begin{enumerate}
  \item \label{it:u1} If $\alpha \le \beta$, then every
    $\beta$-class is the union of
    $\# (\beta: \alpha)$ distinct $\alpha$-classes; put differently, for every
    $a \in A$
    we have $|\{x / \alpha  \mid x \in a/\beta \}| = \#(\beta:\alpha)$.
  \item \label{it:u2} If $\alpha \le \beta \le \gamma$, then
        $\# (\gamma:\alpha) = \# (\gamma:\beta) \, \cdot \, \#(\beta:\alpha)$.
  \item \label{it:u3}
        If $\Int{\alpha}{\beta} \nearrow \Int{\gamma}{\delta}$, then
        $\# (\delta:\gamma) = \#(\beta:\alpha)$.
  \end{enumerate}
\end{lem}
\begin{proof}
  \eqref{it:u1} Each $\alpha$-class contains $|\ab{A}|\, / \,|\ab{A}/\alpha|$ elements,
and each $\beta$-class contains $|\ab{A}|\,/\,|\ab{A}/\beta|$ elements.
Since every $\beta$-class is a disjoint union of $\alpha$-classes, we find
that every $\beta$-class must then consist of exactly
$(|\ab{A}|\,/\,|\ab{A}/\beta|)\, / \, (|\ab{A}|\,/\, |\ab{A}/\alpha|)
= |\ab{A}/\alpha| \,/\, |\ab{A}/\beta| = \#(\beta:\alpha)$
different  $\alpha$-classes.
Property \eqref{it:u2} follows directly from~\eqref{eq:dequot}.
For proving \eqref{it:u3}, we first choose an $a \in A$.
By item~\eqref{it:u1}, it is sufficient to show
that $\{ x / \alpha \mid x \in a / \beta \}$
has the same number of elements as
$\{ x / \gamma \mid x \in a / \delta \}$.
To this end, we define $f \colon 
\{x / \alpha  \mid x \in a/\beta \} \to
\{x / \gamma  \mid x \in a/\delta \}$ by
$f (x / \alpha) = x / \gamma$.  The function $f$ is well-defined because
$\alpha \le \gamma$. For injectivity, let $x,y \in a/\beta$ with
$x/\gamma = y/\gamma$. Then $(x,y) \in \beta \meet \gamma = \alpha$, and thus
$x/ \alpha = y/\alpha$. For surjectivity, we let $y \in a/\delta$.
By congruence permutability, we have
$\delta = \beta \join \gamma = \gamma \circ \beta$, and
therefore there exists $z \in A$ with
$(y,z) \in \gamma$ and $(z, a) \in \beta$.
Then $f(z/\alpha) = z/\gamma = y/\gamma$.
Therefore, $f$ is bijective, which establishes~\eqref{it:u3}.
\end{proof}

\begin{lem} \label{lem:uniform2}
  Let $\ab{A}$ be a finite congruence uniform algebra in a congruence permutable
  variety that is
  the direct product of two algebras $\ab{B}$ and $\ab{C}$ of
  coprime order. Then this product is skew-free.
\end{lem}
\begin{proof}
  Let $\beta$ and $\gamma$ be the projection kernels of $\ab{A}$ such
  that $\ab{A} / \beta \cong \ab{B}$ and $\ab{A} / \gamma \cong \ab{C}$.
  By \cite[Lemma~IV.11.6]{BS:ACIU}, it is sufficient to prove
  that each congruence $\alpha$ of $\ab{A}$ satisfies
  \begin{equation} \label{eq:anonskew}
    (\alpha \join \beta) \meet (\alpha \join \gamma) = \alpha.
  \end{equation}
  We observe that $\Int{\alpha}{\alpha \join \beta} \searrow
  \Int{\alpha \meet \beta}{\beta}$.
  Since every congruence permutable variety is congruence modular,
  we can use the modular law to obtain
  $\alpha \meet \beta = (\alpha \meet \beta) \join 0_A =
  (\alpha \meet \beta) \join (\gamma \meet \beta) =
  ((\alpha \meet \beta) \join \gamma) \meet \beta$.
  Therefore
  $\Int{\alpha \meet \beta}{\beta} \nearrow \Int{(\alpha \meet \beta) \join \gamma}{1_A}$.
  Applying item \eqref{it:u3} of Lemma~\ref{lem:uniform2}, we obtain
  $\# (\alpha \join \beta : \alpha) = \# (1_A: {(\alpha \meet \beta) \join \gamma})$,
  which by item~\eqref{it:u2} of the same lemma divides
  $\# (1_A : \gamma) = |\ab{A} / \gamma|$.
  Hence, using~\eqref{it:u2} again, we have
  $\# ((\alpha \join \beta) \meet (\alpha \join \gamma) : \alpha) \mid |\ab{A}/\gamma|$.
  Changing the roles of $\beta$ and $\gamma$, we obtain
  $\# ((\alpha \join \beta) \meet (\alpha \join \gamma) : \alpha) \mid |\ab{A}/\beta|$.
  Now since $|\ab{A}/\beta|$ and $|\ab{A}/\gamma|$ are coprime,
  we obtain $\# ((\alpha \join \beta) \meet (\alpha \join \gamma) : \alpha) = 1$,
  which implies~\eqref{eq:anonskew}. 
\end{proof}

\section{Proof of the main results} \label{sec:proofs}
    
   \begin{proof}[Proof of Theorem~\ref{thm:1}]
   The items~\eqref{it:t0} and~\eqref{it:t1} are equivalent by definition.

   \eqref{it:t1}$\Rightarrow$\eqref{it:t2}:
  We assume that that the interval $\mathcal{C} = \ob{I} [\Pol (\ab{A}), \Comp (\ab{A})]$
  in the clone lattice is infinite.
  By \cite[Theorem~5.3]{Ai:CMCO}, the set $(\mathcal{C}, \subseteq)$ satisfies the descending
  chain condition, and therefore, there is 
  $C \in \C$ which is minimal such that $\ob{I}[\Pol (\ab{A}),  C]$ is infinite.
  We prove that $C$ is not finitely generated. 
Seeking a contradiction, 
assume that $C$ is finitely generated.
We call a clone $D$ a \emph{subcover} of $C$ if $D \subset C$ and there is no clone
$D'$ with $D \subset D' \subset C$. 
Then by \cite[Chrakterisierungssatz~4.1.3(i)$\Rightarrow$(iii)]{PK:FUR}, $C$
has only finitely many subcovers $D_i$, $i \in I$, and for each clone $E$ on $A$
with $E \subset C$ there is $i \in I$ with $E \subseteq D_i$.
Let
$J := \{j \in I \mid \Pol (\ab{A}) \subseteq D_j \}$. Then 
$\ob{I} [\Pol (\ab{A}), C] = \{C\} \cup \, \bigcup_{j \in J} \ob{I} [\Pol (\ab{A}), D_j]$.
Hence one interval $\ob{I} [\Pol (\ab{A}), D_j]$ must be infinite,
contradicting the minimality of~$C$.

    \eqref{it:t2}$\Rightarrow$\eqref{it:t1}:
    Let $m$ be the maximal arity of the fundamental operations on $\ab{A}$,
    and let $C$ be a nonfinitely generated clone with $\Pol (\ab{A}) \subseteq C
    \subseteq \Comp (\ab{A})$. For $n \ge m$, let
    Let $C_n$ be the subclone of $C$ generated by its $n$-ary members.
    Then $C_m \subseteq C_{m+1} \subseteq \cdots$ and 
    $\bigcup_{n \ge m} C_n = C$. Since $C$ is not finitely generated,
    we have $C_n \subset C$ for all $n \ge m$, and therefore
    the set $\{C_n \mid n \in \N, n \ge m\}$ is infinite.

    Before proving the equivalence with \eqref{it:t3} and \eqref{it:t4},
    we additionally assume that $\ab{A}$ is isomorphic to $\ab{A}_1 \times \cdots \times \ab{A}_n$, and we also assume
     that
      for each $i \in \{1,\ldots, n\}$,
       $\ab{A}_i$ is nilpotent and  
      $|\ab{A}_i|$ is a  prime power, and that
      for all $i,j \in \{1,\ldots, n\}$ with
      $i \neq j$, we have  $\gcd (|\ab{A}_i|, |\ab{A}_j|) = 1$.
    Since $\ab{A}$ has few subpowers, 
    $\ab{A}$ generates a congruence modular variety
    \cite[Theorem~4.2]{BI:VWFS}.
    Thus all the algebras
    $\ab{A}_i$ are nilpotent algebras in a congruence modular variety.
    Representing the congruence $1_A$ of $\ab{A}$ as the join
    of the projection kernels and
    using the join distributivity of the binary commutator
    \cite[Proposition~4.3]{FM:CTFC} to compute
    the lower central series of $\ab{A}$, we see that
    then
    $\ab{A}$ is nilpotent, too.
    Hence by Theorem~6.2 of \cite{FM:CTFC},
    $\ab{A}$ has a Mal'cev term, which we will denote by~$d$, and therefore
    $\ab{A}$ generates a congruence permutable variety \cite{Ma:OTGT}.
    By \cite[Corollary~7.5]{FM:CTFC} $\ab{A}$ and its homomorphic images $\ab{A}_1, \ldots, \ab{A}_n$
    are all congruence uniform.
    Now Lemma~\ref{lem:uniform2} implies 
    that for every $i \in \{1,\ldots, n\}$,
    $\ab{A}$ is a skew-free product $\ab{A}_i \times \ab{C}$ with
    $\ab{C} := \prod_{j \in \{1,\ldots, n\} \setminus \{i\}}  \ab{A}_j$.

\eqref{it:t2}$\Rightarrow$\eqref{it:t4}:
We proceed by contraposition. We assume
that the congruence lattice $\Con (\ab{A})$ does not split strongly
and show that every clone in $\Int{\Pol (\ab{A})}{\Comp (\ab{A})}$ is finitely generated.
We will need another notion of splitting: we say that
a lattice \emph{splits} if it is the union of two proper subintervals;
this definition differs from ``\emph{splits strongly}'' in that
``\emph{splitting}'' does not claim that the  subintervals intersect \cite[p.~861]{AM:SOCO}. Hence
$\ob{L}$ \emph{splits} if
\begin{equation} \label{eq:splits}
   \ob{L} \models \exists \, \delta, \varepsilon \in \ob{L} :
\big(
0 < \varepsilon \text{ and } \delta < 1 \text{ and }
(\forall \alpha \in \ob{L} : \alpha \le \delta \text{ or }
\alpha \ge \varepsilon) \big).
\end{equation}
Let $C$ be a clone with $\Pol (\ab{A}) \subseteq C \subseteq \Comp (\ab{A})$,
and let $\ab{A'} := \algop{A}{C}$ be the corresponding congruence
preserving expansion of $\ab{A}$.
Since the lattice $\Con (\ab{A})$ does not split strongly,
\cite[Corollary~3.4(2)]{AM:SOCO} yields that $\ab{A'}$ is isomorphic
to a direct product
$\ab{B'} \times \ab{C}'_1 \times \ldots  \times \ab{C}'_n$ such that
$\Con (\ab{B'})$ does not split, $n \in \N_0$, and each $\ab{C}'_i$ is simple.
We will now show that each of these direct factors has a finitely
generated clone of polynomial functions. Let us first
examine $\ab{B'}$. The congruence lattice of $\ab{B'}$ does not split, hence
by~Propositions~3.7 and 3.8 of \cite{ALM:FGOC}, $\Pol (\ab{B'})$
is finitely generated. Examining the factors $\ab{C}'_i$, we let
$i \in \{1, \ldots, n\}$ and observe that $\ab{C}'_i$ is a simple finite algebra with Mal'cev term.
If  $\ab{C}'_i$ is abelian, 
$\ab{C}'_i$ is polynomially equivalent to
a module over a ring. Hence its clone of polynomial functions is
generated by its binary members.
If $\ab{C}'_i$ is nonabelian,  then 
$\Pol(\ab{C}'_i)$ consists of
all finitary operations on $\ab{C}'_i$
(this follows, e.~g., from \cite[Corollary~3.5]{HH:ALEC})
and is therefore generated by
its binary members by \cite[p.\ 180]{Po:ITAG} (cf. \cite{Si:SLFD}).
Since by \cite[Corollary~3.4]{AM:SOCO}, the direct product
$\ab{B'} \times \ab{C}'_1 \times \ldots  \times \ab{C}'_n$ has
no skew congruences, we may use Lemma~\ref{lem:abc} $n$ times to obtain
$\Pol (\ab{B'} \times \ab{C}'_1 \times \ldots  \times \ab{C}'_n) = 
 \Pol (\ab{B'}) \otimes \Pol (\ab{C}'_1) \otimes \ldots  \otimes \Pol (\ab{C}'_n)$.
 Now Lemma~\ref{lem:X} implies that $\Pol (\ab{B'} \times \ab{C}'_1 \times \ldots  \times \ab{C}'_n)$
 is finitely generated, and thus also $\Pol (\ab{A'})$ is finitely generated.
 Since $\Pol (\ab{A'}) = C$, $C$ is finitely generated.

 \eqref{it:t4}$\Rightarrow$\eqref{it:t3}:
 From Lemma~\ref{lem:uniform2}, we obtain that for each
 $i \in \{1,\ldots, n-1\}$, the direct product
 $\ab{B} \times \ab{C}$ with $\ab{B} := \ab{A}_i$ and
 $\ab{C} := \prod_{j=i+1}^n \ab{A}_i$ is skew-free.
 Hence we obtain that 
 $\Con (\ab{A})$ is isomorphic to the lattice
 $\prod_{i=1}^n \Con (\ab{A}_i)$. Now
 Lemma~\ref{lem:kl} yields that there is $i \in \{1,\ldots, n\}$
 such that $\Con (\ab{A}_i)$ splits strongly.

 \eqref{it:t3}$\Rightarrow$\eqref{it:t2}:
 Let $i \in \{1,\ldots, n\}$ be such that
 $\Con (\ab{A}_i)$ splits strongly.
 The first part of the proof will produce a nonfinitely generated clone
 $D$
 between $\Pol (\ab{A}_i)$ and $\Comp (\ab{A}_i)$. From $D$,
 it will then be easy to produce a nonfinitely generated clone
 between $\Pol (\ab{A})$ and $\Comp (\ab{A})$.

 In order to produce such a clone $D$, we
 let $\ab{B} := \ab{A}_i$, and $\ab{C} := \prod_{j \neq i} \ab{A}_i$.
 Let $\delta, \varepsilon \in \Con (\ab{B})$ be two congruences witnessing that
 $\Con (\ab{B})$ splits strongly as in~\eqref{eq:splitsstrongly};
 we may choose them in such a way
that $\varepsilon$ is an atom of $\Con (\ab{B})$.
Let $(a, b) \in \varepsilon$ with $a \neq b$,
and for every $n \in \N$, let $f_n  \colon  B^n \to B$ be defined
by
\begin{equation*}
  \begin{array}{rcl}
    f_n (\vb{x}) &=& b \text{ if } \vb{x} \in (B \setminus (a/\delta))^n,
    \text{ and} \\
    f_n (\vb{x}) &=& a \text{ if there exists an } i \in \{1,\ldots, n\}
    \text{ such that } x_i \in a/\delta.
  \end{array}
\end{equation*}
The function $f_n$ is congruence preserving; to this end,
let $\vb{x}, \vb{y} \in B^n$ and let $\alpha$ be a congruence of $\ab{B}$
such that for all $i$, $(x_i, y_i) \in \alpha$.
If $\alpha \le \delta$, then $f_n (\vb{x}) = f_n (\vb{y})$, and therefore
$(f_n (\vb{x}), f_n (\vb{y})) \in \alpha$. If $\alpha \not\le \delta$,
then by the splitting property, $\alpha \ge \varepsilon$. Since
$(f_n (\vb{x}), f_n (\vb{y})) \in \{ (a,a), (a,b), (b,a), (b,b) \} \subseteq \varepsilon$,
we obtain $(f_n (\vb{x}), f_n (\vb{y})) \in \alpha$. Hence $f_n$ is indeed
congruence preserving.
Now we define $D$. To this end, let $\ab{B'}$ be the expansion of $\ab{B}$ with the operations $\{ f_i \mid i \in \N \}$, and
let $D := \Pol (\ab{B'})$.

Our goal is to show that $D$ is not finitely generated. To this end, we first
show that
\begin{equation} \label{eq:Bprimenilpotent}
  \ab{B'} \text{ is nilpotent.}
\end{equation}
For this purpose, we show that $\ab{B'}/\varepsilon$ is nilpotent, and that
$\varepsilon$ is central in $\ab{B'}$.

For the first claim, we observe that $\ab{B'}/\varepsilon$ is an expansion of
$\ab{B} / \varepsilon$ with constant operations because
all $f_i$ have their range contained in one single $\varepsilon$-class and are
therefore constant modulo $\varepsilon$. This implies
$\Pol (\ab{B'} / \varepsilon) = \Pol ( \ab{B} / \varepsilon )$.
Since $\ab{B}/\varepsilon$
is nilpotent, then so is $\ab{B'}/\varepsilon$.

For proving the centrality of $\varepsilon$, we use the relational description
of centrality given in \cite[Proposition~2.3 and Lemma~2.4]{AM:PCOG}, which goes
back to Theorem~3.2(iii) of \cite{Ki:TROT}. From these results, we see that
$\varepsilon$ is central in $\ab{B'}$ 
if and only if
all fundamental operations of $\ab{B'}$ preserve the relation
\[
\rho := \{ (x_1,x_2,x_3,x_4) \in B^4 \mid (x_1, x_2) \in \varepsilon, \,
           d(x_1, x_2, x_3) = x_4 \},
           \]
           where $d$ is the Mal'cev term of $\ab{B}$ that we produced before proving the implication
           \eqref{it:t2}$\Rightarrow$\eqref{it:t4}.           
We will first show that all $f_n$ preserve $\rho$.
To this end, let $n \in \N$ and let
$\langle (x_1^{(i)}, x_2^{(i)}, x_3^{(i)}, x_4^{(i)})  \mid i \in \{1,\ldots, n\} \rangle
\in \rho^n$, and 
for $i \in \{1, \ldots, 4 \}$, set
$y_i := f (x_i^{(1)}, \dots, x_i^{(n)})$.
We have to show
$(y_1, y_2, y_3, y_4) \in \rho$.
Since $f$ is congruence preserving, we have $(y_1, y_2) \in \varepsilon$. The second
property that we have to show is $d(y_1,y_2,y_3) = y_4$.
We first observe that for all $i \in \{1,\ldots, n\}$, we have $(x_1^{(i)}, x_2^{(i)}) \in \varepsilon$
and therefore, since $\varepsilon \le \delta$, also $(x_1^{(i)}, x_2^{(i)}) \in \delta$.
Thus $f_n (x_1^{(1)}, \ldots, x_1^{(n)}) = f_n (x_2^{(1)}, \ldots, x_2^{(n)})$.
Hence $y_1 = y_2$ and therefore 
$$d (y_1,
     y_2,
     y_3) = y_3.$$
    Since for each $i$,
    $x_3^{(i)} = d (x_2^{(i)}, x_2^{(i)}, x_3^{(i)}) \equiv_{\delta} 
    d (x_1^{(i)}, x_2^{(i)}, x_3^{(i)})$,
    and since $f_n$ is constant on $\delta$-classes, we have
    $y_3 =  f_n (x_3^{(1)}, \ldots, x_3^{(n)}) =
    f_n ( \langle  d (x_1^{(i)}, x_2^{(i)}, x_3^{(i)}) \mid i \in \{1,\ldots, n\} \rangle)
    =
    f_n (x_4^{(1)}, \ldots, x_4^{(n)}) = y_4$.
    Therefore, $f_n$ preserves $\rho$.
    In $\ab{B}$, the commutator $[\varepsilon, 1]$ is  $0_B$ because $\ab{B}$ is nilpotent
    and $\varepsilon$ is a minimal congruence of $\ab{B}$.
    Therefore, the relational
    description of centrality implies that every fundamental operation of $\ab{B}$
    preserves $\rho$.
    Hence every fundamental operation of $\ab{B'}$ preserves $\rho$; this implies
    that $\varepsilon$ is central in $\ab{B'}$.
    Since $\ab{B'} / \varepsilon$ nilpotent and $\varepsilon$ is central,
    $\ab{B'}$ is nilpotent, which concludes the proof of~\eqref{eq:Bprimenilpotent}.
    
    Now suppose that $D$ is finitely generated by some finite subset $X$
    of $D$. Then the algebra
    $\ab{B''} := \algop{B}{X}$ satisfies $\Pol (\ab{B''}) = D = \Pol (\ab{B'})$.
    Therefore, $\ab{B''}$ is nilpotent, of finite type and of prime power order.
    Hence, using \cite[Theorem~3.14(3)$\Rightarrow$(4)]{Ke:CMVW}, we obtain that there is a $k \in \N$ such that
    every commutator term of $\ab{B''}$ is of rank at most $k$.
    However,
    \begin{multline*}
        w(x_1,\ldots, x_{k+2}) := \\
     d\Big(
    f_{k+1}  \big( d(x_1, x_{k+2}, a), d(x_2, x_{k+2}, a), \dots
    d(x_{k+1}, x_{k+2}, a) \big),
    a, x_{k+2}
    \Big)
    \end{multline*}
    lies in $\Pol (\ab{B'}) = D = \Clo_B (X) = \Clo (\ab{B''})$.
    Since
    $w (z, x_2, \ldots, x_k, x_{k+1}, z) = \dots
     = w(x_1, x_2, \ldots, x_k ,     z, z) = z$  for all
    $(\vb{x}, z) \in B^{k+2}$, $w$ 
    is a commutator term of $\ab{B''}$. Let $c \not\in a/\delta$. Then
    $w(c,\ldots, c, a) = d(b,a,a) = b$, and therefore $w$ is not the projection
    to the last component, Hence $w$ is a  nontrivial commutator term of rank $k+1$
    in the sense of
    \cite{Ke:CMVW}.
    This contradiction proves that $D$ is not finitely generated.

    From this clone $D$ on $A_i$, we will now produce a clone $E$ on $A$.
    In order to do this,
    we let $E$ be the clone $D \otimes \Pol (\ab{C})$ on $A$.
    Since $\ab{A}$ is a skew-free product,
    Lemma~\ref{lem:abc} implies $\Pol (\ab{A}) = \Pol (\ab{B}) \otimes \Pol (\ab{C})
    \subseteq D \otimes \Pol (\ab{C}) \subseteq \Comp (\ab{B}) \otimes \Comp (\ab{C})$.
    Since $\ab{A}$ is skew-free,
    $\Comp (\ab{B}) \otimes \Comp (\ab{C}) = \Comp (\ab{B} \times \ab{C})  = \Comp (\ab{A})$.
    Hence $E$ is a clone in the interval $\Int{\Pol (\ab{A})}{\Comp (\ab{A})}$.
    Since $D$ is not finitely generated, Corollary~\ref{cor:fg} implies
    that $E$ is not finitely generated.
\end{proof}

\begin{proof}[Proof of Corollary~\ref{cor:pgroup}]
  Since a finite nilpotent group is the direct product
  of its Sylow-subgroups, the result is an instance of the
  equivalence \eqref{it:t0}$\Leftrightarrow$\eqref{it:t4} from
  Theorem~\ref{thm:1}.
\end{proof}

\begin{proof}[Proof of Corollary~\ref{cor:abgroup}]
  From Theorem~1 of \cite{BC:AGWS}, we know that
  the subgroup lattice
  splits (as defined in \eqref{eq:splits})
  iff $r=1$ or ($r \ge 2$ and $m_1 > m_2$).
  
  For the ``if''-direction, let us assume that
  ($r \ge 2$ and $m_1 = m_2$) or ($r=1$ and $m_1 = 1$).
  In the case that $r \ge 2$ and $m_1 = m_2$, the description above tells that
  $\ob{S}$ does not split strongly.
    In the case $r = 1$ and $m_1 =  1$, $\ob{S}$ is a two element chain,
    which does not split strongly, either.
   Hence the implication \eqref{it:t0}$\Rightarrow$\eqref{it:t4} of Theorem~\ref{thm:1}
   yields the result.

   For the ``only if''-direction, we assume that
   $G$ has finitely many polynomially inequivalent expansions.
   We use Theorem~\ref{thm:1} and obtain that $\ob{S}$ does
   not split strongly.
   In the case $r=1$ we obtain that $\ob{S}$ is a chain with
   $m_1 + 1$ elements. Since $\ob{S}$ does not split strongly,
   we then must have $m_1 = 1$.
   We now consider the case $r \ge 2$.
   Seeking a contradiction, we assume $m_1 > m_2$.
   By~\cite{BC:AGWS}, $\ob{S}$ then splits. Lemma~2.1 from \cite{AM:SOCO} describes
   lattices that do split, but not strongly. This lemma yields that
   $\ob{S}$ is isomorphic to a direct product $\ob{M} \times \ob{L}$
   of two lattices such that $\ob{M}$ does not split,
   and $\ob{L}$ is a Boolean lattice.
   Since $\ob{S}$ splits and $\ob{M}$ does not split, we have $|\ob{L}| > 1$.
   Also $|\ob{M}| > 1$:  if $\ob{M}$ has one element, then $\ob{S}$ is Boolean.
   But since $r \ge 2$, $G$ has a subgroup
   isomorphic to $\Z_p \times \Z_p$, and the subgroups of $\Z_p \times \Z_p$ 
   form a nondistributive lattice, contradicting that $\ob{S}$ is Boolean.
   From the lattice isomorphism $\gamma  \colon  \ob{S} \to \ob{M} \times \ob{L}$
   we obtain that
   $G$ is isomorphic to the direct product $H \times K$ of its two non-trivial groups
   $H = \gamma^{-1} (1_{\ob{M}}, 0_{\ob{L}})$ and
   $K = \gamma^{-1} (0_{\ob{M}}, 1_{\ob{L}})$ and that $H \times K$ is a skew-free product
   of $H$ and $K$, meaning that for every subgroup $I$ of $H \times K$, we have
   \begin{equation} \label{eq:is}
     I = (I \cap (H \times \{0\})) + (I \cap (\{0\} \times K)).
   \end{equation}  
   Taking minimal subgroups $H_1$ of $H$ and $K_1$ of $K$,
   we see that $H_1 \times K_1$ is isomorphic to $\Z_p \times \Z_p$,
   and therefore $H_1 \times K_1$ contains $p-1$ skew subgroups $I$ of $H \times K$
   that do not satisfy~\eqref{eq:is}.   
   This contradicts the fact that $H \times K$ is a skew-free product of $H$ and $K$.
   Hence the assumption $m_1 > m_2$ leads to a contradiction, proving that $m_1 = m_2$.
\end{proof}

\section*{Acknowledgments}
The authors thank C.\ Pech and N.\ Mudrinski for discussions on the topics of
this paper. 
\bibliographystyle{alpha}
\newcommand{\etalchar}[1]{$^{#1}$}
\def\cprime{$'$}

\end{document}